\documentclass[11pt,twoside]{amsart}
 \usepackage[T1]{fontenc}

\input{xy}
\xyoption{all}
\xyoption{poly}
\usepackage[all]{xy}

 \usepackage{amsmath,amsthm,amsfonts,amssymb}
      \theoremstyle{plain}
      \newtheorem{theorem}{Theorem}[section]
      \newtheorem*{theorem*}{Theorem}
      
      \newtheorem{corollary}[theorem]{Corollary}
      \newtheorem{proposition}[theorem]{Proposition}

      \theoremstyle{definition}
	  \newtheorem{example}[theorem]{Example}
      \newtheorem{definition}[theorem]{Definition}
      
     \theoremstyle{remark}
      \newtheorem{remark}[theorem]{Remark}

\def\hocolim{\operatorname{\underline{hocolim}} }
\def\thocolim{\operatorname{hocolim}}

 \def\C{{\mathcal C}}

 \def\K{{\mathcal K}}

 \def\N{{\mathcal N}}
 \def\P{{\mathcal P}}

 \def\X{{\mathcal X}}

\newcommand\se{{\hspace{2 pt}\diagup\hspace{-4.8 pt} \searrow\hspace{5 pt}}}
\newcommand\co{{\hspace{2 pt}\searrow \hspace{3 pt}}}

\newcommand\we{\simeq\hspace {-11 pt}_{_{_{we}}}}

      \makeatletter
      \def\@setcopyright{}
      \def\serieslogo@{}
      \makeatother
   
\begin{document}

\title [Homotopy colimits of diagrams over posets]{Homotopy colimits of diagrams over posets and variations on a theorem of Thomason}
   \author{Ximena Fern\'andez}
   \author{El\'ias Gabriel Minian}
   \address{Departamento  de Matem\'atica - IMAS\\
 FCEyN, Universidad de Buenos Aires. Buenos Aires, Argentina.}
\email{xfernand@dm.uba.ar ; gminian@dm.uba.ar}

   \begin{abstract}
   We use a classical result of McCord and reduction methods of finite spaces to prove a generalization of Thomason's theorem on homotopy colimits over posets. In particular this allows us to characterize the homotopy colimits of diagrams of simplicial complexes in terms of the Grothendieck construction on the diagrams of their face posets. We also derive analogues of well known results on homotopy colimits in the combinatorial setting, including a cofinality theorem and a generalization of Quillen's Theorem A for posets.
   \end{abstract}



   \maketitle

   \section{Introduction}

Let $\C$ be a small category and let  $X:\C\rightarrow \mathrm{CAT}$
be a functor to the category $\mathrm{CAT}$ of small categories. Recall that 
the \emph{Grothendieck construction} on $X$, which is usually denoted by $\C\int X$,
is the following category. The objects are the pairs $(c, x)$,
where $c$ is an object of $\C$ and $x$ is an
object of $X(c)$, and the morphisms $(\alpha,\beta): (c,x) \rightarrow (c',x')$ are given by morphisms
$\alpha: c \rightarrow c'$ in $\C$ and $\beta: X(\alpha)(x) \rightarrow x'$ in $X(c')$. Thomason's theorem \cite{Th} establishes  the existence of  a natural homotopy equivalence
\[\thocolim\hspace{5 pt}\N X\rightarrow \N(\C\int X)\] from 
the homotopy colimit of the nerve of $X$ to the nerve of the Grothendieck construction.

In this article we focus our attention on homotopy colimits of diagrams of spaces indexed by finite partially ordered sets. The main idea is to use the interaction between the combinatorics and the topology of finite topological spaces to investigate homotopy colimits of diagrams of polyhedra. If $P$ is a finite poset and $X:P\to  \P_{<\infty}$ takes values in the category of finite posets, the Grothendieck construction $P\int X$ is also a finite poset and, in fact, it is very simple to characterize. On the other hand any finite poset can be regarded as a finite topological space where the open subsets are its downsets (see for example \cite{Ba1,BM1,Ma, Ma2,Mc}). We handle the Grothendieck construction on a diagram of finite posets as a finite topological space and use a local-to-global theorem of McCord \cite[Thm 6]{Mc} to derive analogues of well known results on homotopy colimits in the combinatorial setting and to prove a generalization of Thomason's theorem. This generalization allows us to apply combinatorial methods to investigate homotopy colimits of diagrams of polyhedra (indexed by finite posets).  In \cite{Ba2} Barmak exhibited a very simple proof of Quillen's Theorem A for posets \cite{Qu} (or equivalently, McCord's theorem for finite topological spaces \cite{Mc}) using the \it non-Hausdorff mapping cylinder \rm $B_f$ of a poset map $f:X\to Y$. The non-Hausdorff mapping cylinder is a finite analogue of the classical mapping cylinder of a continuous function, and similarly to its classical version, $B_f$ is the \it non-Hausdorff homotopy colimit \rm (i.e. the Grothendieck construction) of the diagram of posets $\xymatrix@1{X\ar[r]^f & Y}$, indexed by the poset $\bf{1}$ of two elements $0<1$. 

In Section 2 we study non-Hausdorff homotopy colimits of $P$-diagrams from the finite space point of view.  We use \it reduction methods \rm  to investigate their weak homotopy types. Quillen's Theorem A for posets follows immediately from Proposition \ref{dbp} and Proposition \ref{ubp} below, by applying the results to the poset $\bf{1}$.

The main result of the article is the following generalization of Thomason's theorem in the context of finite posets.

\begin{theorem*}
Let $P$ be a finite poset. Let $K:P\to \mathcal{S}$ be a diagram of spaces and $X:P\to \P_{< \infty}$ be a diagram of finite posets. Let $\phi:K\to X$ be a diagram morphism (where $X$ is viewed as a diagram of finite topological spaces) such that $\phi_p:K_p\to X_p$ is a weak homotopy equivalence for every $p\in P$.  Then there exists a weak homotopy equivalence $$\hat\phi:\thocolim K\to \hocolim X$$
from the homotopy colimit of $K$ to the non-Hausdorff homotopy colimit of $X$ (viewed as a finite topological space).
\end{theorem*}

As an immediate consequence of this result we derive Thomason's theorem in the context of posets, and also a kind of converse of Thomason's theorem, which relates the homotopy colimit of a diagram of simplicial complexes with the non-Hausdorff homotopy colimit of the diagram of their face posets. In combination with the reduction methods of Section 2, this allows us to simplify the computation of homotopy colimits of diagrams of spaces. 

It is well known that for any simplicial set $T$, there is a \it natural \rm homotopy equivalence 
$sd\ T \to T$ from the (geometric realization of the) barycentric subdivision of $T$ to $T$ (see for example \cite[Thm 12.2.5]{Ma3}). By Bousfield-Kan's homotopy lemma \cite{BK}, this implies that the homotopy colimit of a diagram of simplicial sets is homotopy equivalent to the homotopy colimit of the diagram of their barycentric subdivisions. On the other hand, any ordered simplicial complex $K$ (i.e. a simplicial complex together with a partial ordering of its vertices that restricts to a total ordering on the vertices of each simplex) can be seen as a simplicial set $K_s$. Moreover, the simplicial set associated to its (geometric) barycentric subdivision $(K')_s$ is naturally isomorphic to $sd\ K_s$, the subdivision of the simplicial set $K_s$ (see \cite[Thm 12.2.2]{Ma3}). This proves that the homotopy colimit of a diagram of ordered simplicial complexes (and ordered simplicial maps) is homotopy equivalent to the homotopy colimit of the diagram of their barycentric subdivisions. However, in a general geometric situation, one has to deal with diagrams of unordered simplicial complexes, and in the unordered context there is no natural homotopy equivalence between $K$ and its barycentric subdivision $K'$ (although their underlying topological spaces are equal). Our methods turn out to be appropriate to circumvent this problem. As a corollary of our main theorem, we prove invariance of homotopy type under barycentric subdivision for homotopy colimits in the (general) unordered setting.

The paper concludes with some new results on homotopy colimits of diagrams of polyhedra, which are obtained by applying the main theorem together with the combinatorial methods introduced in Section 2.

\section{The Grothendieck construction on posets and reduction methods}\label{methods}

Given a finite poset $X$, we denote by $\K(X)$ its classifying space. $\K(X)$ is also called the order complex of $X$ and it is the simplicial complex whose simplices are the non-empty chains of $X$. A finite poset can be seen as a finite topological space whose open subsets are the downsets (see for example \cite{Ba1,BM1,Ma, Ma2, Ma3, Mc}). The topology of the finite space $X$ is related to the topology of the classifying space $\K(X)$. This was studied by McCord in \cite{Mc}. Concretely, there is a natural weak homotopy equivalence $\mu:\K(X)\to X$ (i.e. a continuous map which induces isomorphisms in all homotopy groups). $\mu$ is called \it the McCord map \rm and it is defined as follows. Given $\alpha \in \K(X)$, write $\alpha=\sum_{i=1}^r t_i x_i$ with $\sum_{i=1}^r t_i=1$ and $t_i>0$, where $x_1<x_2<\cdots<x_r$ is 
a chain of $X$, and define $\mu(\alpha)=x_1$. 

Throughout this paper we will handle finite posets as finite topological spaces and use the weak equivalence $\K(X)\to X$. A function $f:X\to Y$ between finite posets (=finite spaces) is a poset map (i.e. it is order preserving) if and only if it is continuous. Note that a poset $X$ is weakly contractible (= homotopically trivial) if and only if $\K(X)$ is a contractible polyhedron. More generally, two finite posets $X$ and $Y$ are weak equivalent (denoted by $X\we Y$) if and only if $\K(X)$ and $\K(Y)$ are homotopy equivalent, and $f:X\to Y$  is a weak equivalence if and only if its realization is a homotopy equivalence. A finite poset is called a \it finite model \rm of a CW-complex $K$ if its classifying space $\K(X)$ is homotopy equivalent to $K$.

Given a finite simplicial complex $K$, its face poset will be denoted by $\X(K)$. This is the poset of simplices of $K$ ordered by inclusion. A simplicial map $f:K\to L$ induces a poset map $\X(f):\X(K)\to\X(L)$. McCord defined in \cite{Mc} a weak equivalence $\tilde\nu:K\to \X(K)$ using the McCord map of above and the identification of $K$ with its barycentric subdivision $K'$. This weak equivalence is however not natural (only up to homotopy). Since we need naturality to apply Theorem \ref{main} below, we will use the following variation of the McCord map $\tilde\nu$, which was introduced by Barmak in \cite{Ba1}. Given a simplicial complex $K$, we consider its face poset $\X(K)$ with the opposite order, denoted by $\X(K)^{op}$, and define a map $\nu:K\to \X(K)^{op}$ by $\nu(x)=\sigma$ if $x\in\overset{\circ}{\sigma}$. Here $\overset{\circ}{\sigma}$ denotes the interior of $\sigma$. By \cite[Thm 11.3.2]{Ba1} $\nu$ is a natural weak equivalence.


 For any $x \in X$,  let
$U_x\subseteq X$ be the subposet of elements which are smaller than or equal to $x$ and let $\hat U_x = U_x\smallsetminus\{x\}$. Analogously, we denote by
$F_x$  the subposet of elements of $X$ which are greater than or equal to $x$ and let $\hat F_x =F_x\smallsetminus\{x\}$. The open subsets $\{U_x\}_{x\in X}$ form a basis for the topology of $X$. When a point $x$ belongs to different posets $X,Y$, we write $U_x^X,U_x^Y,F_x^X,F_x^Y$ to distinguish whether the subposets are considered in $X$ or in $Y$. For $x,y\in X$, we write $x\prec y$ if $x$ is covered by $y$, i.e. if $x<y$ and there is no $z$ such that $x<z<y$. A linear extension of a finite poset $X$ is a total ordering $x_1,\ldots,x_n$ of its elements such that if $x_i\leq x_j$ in $X$ then $i\leq j$.

In \cite{Ba1,BM1,BM2} various reduction methods were introduced. A reduction method allows us to study and handle the homotopy type of a polyhedron by combinatorial moves on their finite models. We describe here some of these methods, which will be used in the rest of this article. For a comprehensive exposition the reader may consult \cite{Ba1,BM1,BM2}. The first reduction method was introduced by Stong \cite{St} (see also \cite{Ma,Ma2}). Given a finite poset $X$, an element $x\in X$ is called an \it up beat point \rm if $\hat F_x$ has a minimum, i.e. there is a unique element $y\in X$ such that $x\prec y$. Analogously $x$ is called a \it down beat point \rm if $\hat U_x$ has maximum (there is a unique $y$ such that $y\prec x$). In both cases we say that $x$ is \it dominated \rm by $y$. If $x$ is a beat point (up or down), $X\smallsetminus\{x\}\subset X$ is a strong deformation retract (as finite topological spaces, and therefore at the level of classifying spaces). A finite poset $X$ is contractible (=dismantlable) if and only if one can remove beat points from $X$, one by one, to obtain a one-point space (singleton). If $X$ is contractible then it is weakly contractible, but the converse does not hold.
The notion of \it collapse \rm in the context of posets was introduced by Barmak and Minian in \cite{BM1} and it corresponds to Whitehead's classical notion of simplicial collapse. A point $x\in X$ is a \it down weak point \rm if $\hat U_x$ is contractible (i.e. dismantlable), and it is an \it up weak point \rm if  $\hat F_x$ is contractible. An \textit{elementary collapse} is the deletion of an up or
 down weak point. The inverse operation is called an \textit{elementary
 expansion}.  We say that $X$ \textit{collapses} to $Y$ (or
$Y$ expands to $X$), and denote $X \co Y$,
 if there is a sequence of elementary collapses which
starts in $X$ and ends in $Y$. A poset $X$ is said to be \it collapsible \rm if it collapses to
 a point. Since any beat point is in particular a weak point, if $X$ is a contractible poset then it is collapsible. In \cite{Ba1,BM1} there are various examples of collapsible posets which are not contractible. Finally, we say that $X$ is \textit{simply equivalent} to $Y$ if there exists a sequence of collapses
and expansions that starts in $X$ and ends in $Y$. This is denoted by 
 $X\se Y$. In \cite{BM1} it is proved that $X\se Y$ if and only if $\K(X)$ and $\K(Y)$ are simply equivalent polyhedra (which, as customary, is denoted by $\K(X)\se \K(Y)$). In particular if $X\se Y$ then they are weak equivalent (and collapsible posets are homotopically trivial). Moreover, if $X\co Y$ then $\K(X)\co \K(Y)$. It can be shown that if $\hat U_x$ or $\hat F_x$ is a collapsible poset then $\K(X)\co \K(X\smallsetminus\{x\})$ (see \cite{BM2}).

More generally, we say that a point $x\in X$  is a \it $\gamma$-point \rm if 
$\hat U_x$ or $\hat F_x$ is homotopically trivial. It is proved in \cite{BM2} that in that case, $X \se X\smallsetminus \{x\}$. Therefore, if a finite poset $X$ can be reduced to a point by removing $\gamma$-points, one by one, then it is homotopically trivial (i.e. its classifying space $\K(X)$ is contractible).



Working with diagrams over finite posets allows us to apply reduction methods to study their homotopy colimits. We use reduction methods to derive old and new results on homotopy colimits.  We refer the reader to \cite{BK} and \cite{Vo} for the basic definitions and results on homotopy colimits of spaces, and to \cite{WZZ} for applications of homotopy colimits to combinatorial problems.

Let $P$ be a finite poset, viewed as a small category with a unique arrow 
 $p\rightarrow q$ for each $p,q\in P$ such that $p\leq q$, and let $X$ be a \textit{$P$-diagram of finite posets}, i.e. a
functor from $P$ to the category $\P_{< \infty}$ of finite
posets. In this case the Grothendieck construction on $X$ can be described as follows.

\begin{definition}[The non-Hausdorff homotopy colimit of finite posets] Let $X:P\rightarrow \P_{< \infty}$ be a functor.
The \emph{non-Hausdorff homotopy colimit of $X$}, denoted by $\hocolim X$, 
is the following poset. The underlying set is the disjoint union $\coprod_{p\in P} X_p$. 
We keep the given ordering within $X_p$ for all $p \in P$,
 and for every  $x\in X_p$ and $y\in X_q$ such that $p\leq q$, we set $x\leq y$ in $\hocolim X$ if $f_{pq}(x)\leq y$ in $X_q$.  Here $X_p=X(p)$ for each $p\in P$ 
 and $f_{pq}=X(p\rightarrow q)$ for each $p\leq q$ in $P$.
\end{definition}

By Thomason's theorem we have a homotopy equivalence
\[\thocolim\hspace{5 pt} \K X \simeq \K(\hocolim X).\]

In the context of finite posets, Thomason's theorem can be deduced from a more general result. This will be proved in the next section.

\begin{remark}
At this point it is worth noting the difference between $\hocolim X$ and $\thocolim X$ for a given diagram $X:P\to \P_{< \infty}$. The first one is the Grothendieck construction on $X$ and it is a finite poset (which can be viewed as a finite topological space). The second one is the classical construction of homotopy colimit of a diagram of topological spaces (applied, in this case, to a diagram of finite topological spaces) and it is not a finite space. However, as an immediate consequence of our main result of next section, we will see that they are weak equivalent spaces (when we view $\hocolim X$ as a finite topological space).
\end{remark}

\begin{example} \label{mapcyl} Any map $f:X_0\rightarrow X_1$ between finite posets can be viewed as a diagram $X:\bf{1}\rightarrow  \P_{< \infty}$ where $\bf{1}$ is the poset
of two elements $0<1$. Similarly as in the topological
context, we have  $\hocolim X= B_f$, 
the non-Hausdorff mapping cylinder of $f$, introduced in \cite{Ba1}.
\end{example}

 It is a well known fact that the mapping cylinder of a continuous function $f:W\to Z$ is homotopy equivalent to $Z$. Similarly, if $f:X\to Y$ is a map of posets, the non-Hausdorff mapping cylinder $B_f$ collapses to $Y$. This result can be viewed as a particular case of Corollary \ref{maximum} below: if the indexing poset $P$ has maximum $p$, the homotopy colimit of any $P$-diagram collapses to $X_p$.
If $f:X\to Y$ is a poset map such that $f^{-1}(U_y)$ is weakly contractible for every $y\in Y$, then $f$ is a weak equivalence (i.e. it induces a homotopy equivalence between the classifying spaces). This is Quillen's Theorem A for posets \cite{Qu} and McCord's theorem for finite topological spaces \cite{Mc}. In Proposition \ref{dbp} we generalize this result for homotopy colimits of $P$-diagrams.
Quillen's Theorem A for posets follows immediately from Propositions \ref{dbp} and \ref{ubp} by applying the results to the poset $\bf{1}$.

Given a poset map $\phi:P\to Q$ and a $Q$-diagram $X$, we denote by $\phi^*X$ the $P$-diagram obtained by pulling back $X$ along $\phi$. Concretely, $\phi^*X=X\phi$. There is a canonical map $\hocolim \phi^*X \to \hocolim X$ induced by the identities $(\phi^*X)_p=X_{\phi(p)}$. If $i:Q'\to Q$ is a subposet and
$X$ is a $Q$-diagram, the restriction $i^*X$ is denoted by $X_{|_{Q'}}$. Note that in this case $\hocolim X_{|_{Q'}}$ is a subposet of $\hocolim X$.

\begin{proposition} \label{ubp} Let $X:P\rightarrow \P_{<\infty}$
be a $P$-diagram of finite posets. If $p\in P$ is an up beat point, then 
$\hocolim X\co \hocolim X_{|_{P\smallsetminus \{p\}}}$. In particular, they are weak equivalent.
\end{proposition}

\begin{proof} Let $x_0, x_1, \cdots, x_n$ be a linear extension of $X_p^ {op}$, the opposite poset of $X_p$. 
Let $Y_0=\hocolim X$ and for each $0\leq i \leq n$ define inductively $Y_{i+1}=Y_{i}\smallsetminus \{x_i\}$. Let $q$ be the minimum element of $\hat F_p^P$. Note that
$\hat F_{x_i}^{Y_i}=F_{f_{pq}(x_i)}^{Y_{i}}$. This implies that  $x_i$ is an up beat point 
of $Y_i$ for all $1\leq i\leq n$ and therefore
$$\hocolim X= Y_0\co Y_1 \co \cdots \co Y_{n+1}= (\hocolim X) \smallsetminus X_p
=  \hocolim X_{|P\smallsetminus \{p\}}.$$
\end{proof}

\begin{corollary} \label{maximum} Let $X:P\rightarrow \P_{<\infty}$
 be $P$-
diagram of finite posets.
If $P$ has a maximum element $p$, then $\hocolim(X)\co X_p$. In particular, they are weak equivalent.
\end{corollary}

\begin{proof}Let $p=p_0, p_1, \cdots, p_n$ be a linear extension of $P^{op}$. Since 
$P$ has a maximum element, 
there is a sequence of collapses $$P\co P\smallsetminus\{p_1\}\co
P\smallsetminus \{p_1, p_2\}\cdots \co P\smallsetminus \{p_1,
p_2, \cdots, p_n\}=\{p\},$$ where $p_i$ is an up beat point of $P\smallsetminus \{p_1, p_2,
\cdots, p_{i-1}\}$.
By applying recursively  Proposition \ref{ubp}, we have 
\[\hocolim X \co (\hocolim X) \smallsetminus X_{p_1}
\co \cdots \co (\hocolim X)\smallsetminus
\bigcup_{i=1}^ n X_{p_i}= X_p\]
\end{proof}

From this result and McCord's theorem \cite[Thm 6]{Mc} we deduce the following analogue of 
Bousfield-Kan's homotopy lemma (cf. \cite{BK, WZZ}). Note that it can also be deduced from the original homotopy lemma for diagrams of spaces and Thomason's theorem.

\begin{corollary}[Homotopy Lemma]\label{homotopy} Let $P$ be a finite
 poset, let $X,Y:P\rightarrow \P_{<\infty}$ be
$P$-diagrams of finite posets and $\alpha:X\rightarrow Y$ a 
morphism of diagrams. 
If $\alpha_p:X_p\rightarrow Y_p$ is a weak equivalence for every 
$p\in P$, then $\alpha$ induces a weak equivalence 
\[\hocolim X \we \hocolim Y .\]
\end{corollary}

\begin{proof}
For any $p\in P$ consider $\hocolim X_{|_{U_p}}$ the homotopy colimit of the diagram restricted to $U_p$. Since $p$ is the maximum of $U_p$, by Proposition \ref{maximum} $\hocolim X_{|_{U_p}} \co X_p$ and therefore $\alpha$ induces a weak equivalence
$\hocolim X_{|_{U_p}} \we \hocolim Y_{|_{U_p}}$. Now the results follows from \cite[Thm 6]{Mc} applied to the basis-like open cover $\{\hocolim Y_{|_{U_p}}\}_{p\in P}$ of  $\hocolim Y$.
\end{proof}

\begin{remark}\label{sonbeat}
Note that all the collapses in Proposition \ref{ubp} and Corollary \ref{maximum} are strong collapses, in the sense that all the points removed are beat points (not just weak points). This implies that if the finite space $X_p$ in Corollary \ref{maximum} is contractible (i.e. it is a dismantlable poset) then so is $\hocolim X$.
\end{remark}

\begin{proposition} \label{dbp} Let $X:P\rightarrow \P_{<\infty}$ be 
a $P$-
diagram of finite posets.
If $p$ is a down beat point of $P$ dominated by an element $q$ 
and $f_{qp}^{-1}(U_x)$ is contractible for every $x \in X_p$, then $\hocolim(X)\co
\hocolim(X_{|P\smallsetminus \{p\}})$. In particular they are weak equivalent.
\end{proposition}

\begin{proof}
Let $x_0, x_1, \cdots, x_n$ a linear extension of $X_p$. 
Define $Y_0=\hocolim X$, and inductively $Y_{i+1}=Y_{i}\smallsetminus \{x_i\}$ for every $0\leq i \leq n$.
We will show that $Y_0 \co Y_1 \co \cdots \co Y_{n+1}=(\hocolim X)\smallsetminus X_p.$

Note that $\hat U_{x_i}^ {Y_i}=\hocolim \tilde X^ i$, where
$\tilde X^ i:\hat U_p^ P \rightarrow \P_{<\infty}$ is the functor defined by
$\tilde X^ i(p')=f_{p'p}^{-1}(U_{x_i})$, for all $p'< p$ in $P$ (where the transition maps are induced by the original transition maps).

Since $\hat U_p^ P$ has a maximum element $q$ and  $f_{qp}^{-1}(U_{x_i})$ is contractible,  by Corollary \ref{maximum} and Remark \ref{sonbeat},   
$\hat U_{x_i}^ {Y_i}$ is a contractible finite space. This proves that $x_i$ is a weak point of $Y_i$. 
\end{proof}

\begin{proposition}\label{dbpgen}  Let $X:P\rightarrow \P_{<\infty}$ be a $P$-
diagram of finite posets.
If $p$ is a down beat point of $P$ dominated by $q$
and $f_{qp}$ is a weak equivalence, then
$$\hocolim X\we \hocolim X_{|_{P\smallsetminus \{p\}}}.$$
\end{proposition}

\begin{proof}
Since $p$ is dominated by $q$, there is a well defined strong homotopy retraction 
$r:P\to P\smallsetminus \{p\}$ which is the identity for any $p'\neq p$ and $r(p)=q$. Denote by $i:P\smallsetminus\{p\}\to P$ the inclusion. There is a morphism of $P$-diagrams $\gamma:(ir)^*X\to X$ such that $\gamma_{p'}$ is the identity for every $p'\neq p$ and  $\gamma_p=f_{qp}:((ir)^*X)_p=X_q\to X_p$. By hypothesis and Corollary \ref{homotopy}, $\gamma$ induces a weak equivalence $\hocolim (ir)^*X\we \hocolim X$, and by Proposition \ref{dbp},
 $\hocolim (ir)^*X\we  \hocolim X_{|_{P\smallsetminus \{p\}}}$.
\end{proof}

\begin{example} By the previous proposition one can immediately deduce that the non-Hausdorff homotopy pushout of the poset diagram
 $$\xymatrix{X_0\ar[r]^{f_1}\ar[d]_{f_2}&X_1\\
		X_2&}$$ is weak equivalent to $X_0$ provided the maps $f_1$ and $f_2$ are weak equivalences.
\end{example}

We prove now an analogue, in the context of posets, of a cofinality theorem of Bousfield and Kan. We need first a generalization of Proposition \ref{ubp}. 

\begin{proposition}\label{up wp}
 Let $X:P\rightarrow 
\P_{<\infty}$ be a $P$-diagram of finite posets. 
If $p$ is a point of $P$ such that $\hat F_p$ is homotopically trivial, then $\hocolim X\we \hocolim X_{|_{P\smallsetminus \{p\}}}$.

\end{proposition}

\begin{proof}
Let $x_1, x_2, \cdots, x_n$ be a linear extension of $X_p^{op}$.
 For each $1\leq i \leq n$, we define the $\hat F_p$-diagram  $\tilde X^ i:\hat F_p \rightarrow \P_{<\infty}$ as follows.  
 $\tilde X^ i(p')=F_{f_{pp'}(x_i)}$ and $\tilde X^ i(p'\rightarrow p'')=
 f_{p'p''|_{F_{f_{pp'}(x_i)}}}:F_{f_{pp'}(x_i)}\rightarrow F_{f_{pp''}(x_i)}$.
 
 Since $\tilde X^ i(p')$ is contractible for all $p'\in \hat F_p$, then 
 by Corollary \ref{homotopy} it follows that $\hocolim \tilde X^ i\we \hat F_p$, and by hypothesis, the last one is weakly contractible. Therefore, since
 \[\hat F_{x_i}^{(\hocolim X) \smallsetminus \{x_1, x_2, \cdots, x_{i-1}\}}
=\hocolim \tilde X^i,
\]
which is homotopically trivial, we can remove the points $x_1,\ldots,x_n$ of $X_p$ one by one, similarly as in the proof of Proposition \ref{ubp}, and the result follows.

\end{proof}

We prove now the main result of this section, which is an analogue of Bousfield-Kan's cofinality theorem in the combinatorial setting.

\begin{theorem}[Cofinality Theorem]
 Let $\varphi:P\rightarrow Q$ an order preserving map between finite posets.
Let $X:Q\rightarrow \P_{<\infty}$ be a $Q$-diagram. If $\varphi^{-1}(F_q)$ is
homotopically trivial for all $q\in Q$, then the canonical map $\hocolim \varphi^* X\rightarrow \hocolim X$
is a weak equivalence.
\end{theorem}

\begin{proof}
Let $R$ be the following poset. The underlying set is the disjoint union $Q\coprod P$. We keep the given ordering within $Q$ and $P$ and for every $q\in Q$ and $p\in P$ we set $q\leq p$ if there are $q'\in Q$ and $p'\in P$ such that $q\leq q'$, $p'\leq p$ and $\varphi(p')=q'$. Consider the  following $R$-diagram $\tilde X: R\rightarrow \P_{<\infty}$. For each $q\in Q$ take $\tilde X(q)=X_q$, set $\tilde X(p)=X_{\varphi(p)}$ for $p\in P$, 
 $\tilde X(p\rightarrow p')=X(\varphi(p)\rightarrow \varphi(p'))$,
 $\tilde X(q\rightarrow q')=X(q\rightarrow q')$, and $\tilde X(\varphi(p)\rightarrow p)=
 \mathrm{id}_{X_{\varphi(p)}}$. Note that the restriction of $\tilde X$ to $Q$ is the original diagram $X$ and the restriction of $\tilde X$ to $P$ is $\varphi^*X$.
 
 Take a linear extension $q_1, q_2, \cdots, q_m$ of $Q^{op}$. For each $1\leq j \leq m$ we have 
 $$\hat F_{q_j}^{R
\smallsetminus \{q_1, q_2, \cdots, q_{j-1}\}}=
\varphi^{-1}(F_{q_j})$$ which is homotopically trivial by hypothesis.
 By applying  recursively Proposition \ref{up wp}, we get 
$\hocolim \tilde X\we \hocolim \varphi^* X$.

Similarly,  take a linear extension $p_1, p_2, \cdots, p_n$ of $P$. For each $1\leq i \leq n$ we have $$\hat U_{p_i}^{R\smallsetminus 
\{p_1, p_2, \cdots, p_{i-1}\}}=
\varphi(U_{p_i})=U_{\varphi(p_i)}.$$ Therefore, $p_i$ is a down beat point of 
$R\smallsetminus 
\{p_1, p_2, \cdots, p_{i-1}\}$ dominated by $\varphi(p_i)$, and $\tilde X(\varphi(p_i)\rightarrow p_i)$ is the identity. By Proposition \ref{dbp} we have $\hocolim \tilde X \we \hocolim X$.
\end{proof}

\section{Variations on Thomason's theorem and applications}

In this section we prove a result that relates the homotopy colimit of a diagram of spaces with the non-Hausdorff homotopy colimit of the diagram of their models. As a consequence we obtain an alternative and simple proof of Thomason's theorem in the context of posets. As another immediate consequence of the main result we deduce that the homotopy colimit of a diagram of finite simplicial complexes is weak equivalent to the non-Hausdorff homotopy colimit of the diagram of their face posets. This implies that all the techniques developed in the previous section for non-Hausdorff homotopy colimits can be used for diagrams of polyhedra (indexed by finite posets) by means of the face poset functor.

We denote by $\mathcal{S}$ the category of topological spaces and continuous maps. Sometimes we require a diagram of spaces $D:P\to \mathcal{S}$ to satisfy extra conditions (for instance $D$ can be a diagram of simplicial complexes, finite topological spaces, etc), however in all these cases the homotopy colimit $\thocolim D$ is taken in the category $\mathcal{S}$. 

\begin{theorem}\label{main}
Let $P$ be a finite poset. Let $K:P\to \mathcal{S}$ be a diagram of spaces and $X:P\to \P_{< \infty}$ be a diagram of finite posets. Let $\phi:K\to X$ be a diagram morphism (where $X$ is viewed as a diagram of finite topological spaces) such that $\phi_p:K_p\to X_p$ is a weak equivalence for every $p\in P$.  Then there exists a weak equivalence $$\hat\phi:\thocolim K\to \hocolim X.$$
\end{theorem}

\begin{proof}

We define first the map $\hat\phi:\thocolim K\to \hocolim X$. For every $p\leq p'$ denote by $f_{pp'}=K(p\to p')$ and $g_{pp'}=X(p\to p')$ the transition maps. Recall that $\thocolim K$ can be constructed from the disjoint union $\coprod_{p\in P} K_p\times \K(F_p)$ by identifying the pairs $(\alpha,\beta)\in K_p\times \K(F_p)$ with $(\alpha',\beta')\in K_{p'}\times \K(F_{p'})$ if $f_{pp'}(\alpha)=\alpha'$ and $\beta=\beta'\in \K(F_{p'})$. We denote by $\sim$ the equivalence relation generated by this identification.
 
For each $p\in P$, let $\mu_p:\K(F_p)\to F_p\subseteq X_p$ be the McCord map (defined at the beginning of the previous section). Given $(\alpha,\beta)\in K_p\times \K(F_p)$ we define $$\hat\phi(\alpha,\beta)=\phi_{\mu_p(\beta)}(f_{p\mu_p(\beta)}(\alpha))\in X_{\mu_p(\beta)}\subseteq \hocolim X.$$
It is easy to verify that $\hat\phi$ is well defined. In order to see that it is a continuous map, it suffices to prove that for each $y\in \hocolim X$,  $\hat\phi^{-1}(U_y^{\hocolim X})\cap (K_q\times \K(F_q))$ is open in $K_q\times \K(F_q)$ for every $q\in P$. Fix $y\in\hocolim X$ and let $p\in P$ such that $y\in X_p$. Note that $\hat\phi^{-1}(U_y^{\hocolim X})\cap (K_q\times \K(F_q))$ is empty if $q\not\leq p$ and it is equal to $(g_{qp}\phi_q)^{-1}(U_y^{X_p})\times \mu_q^{-1}(U_p^{F_q})$ if $q\leq p$. This proves that $\hat\phi$ is continuous.

In order to prove that $\hat\phi$ is a weak homotopy equivalence, we use McCord's theorem \cite[Thm 6]{Mc} for the basis-like open cover $\{\hocolim X|_{U_p}\}_{p\in P}$ of $\hocolim X$. We have to see that
$$\hat\phi:\hat\phi^{-1}(\hocolim X|_{U_p})\to \hocolim X|_{U_p}$$
is a weak equivalence for each $p$. Note that $\hat\phi^{-1}(\hocolim X|_{U_p})=\coprod_{q\leq p}K_q\times \mu_q^{-1}(U_p^{F_q}) / \sim$ and that there is a commutative diagram
$$\xymatrix{
K_p\times \{p\} \ar[r]^{\phi_p}\ar[d]_{i} & X_p\ar[d]^j\\
(\coprod_{q\leq p}K_q\times \mu_q^{-1}(U_p^{F_q}) / \sim) \ar[r]^(.6){\phi} & \hocolim X|_{U_p}}.$$
The inclusion $j$ is a weak equivalence by Corollary \ref{maximum} and $\phi_p$ is a weak equivalence by hypothesis, thus we only need to check that the inclusion $i$ is a homotopy equivalence.

Consider the retraction $r:(\coprod_{q\leq p}K_q\times \mu_q^{-1}(U_p^{F_q}) / \sim)\to K_p\times \{p\}$ defined by $r(\alpha,\beta)=(f_{qp}(\alpha),p)$ for $(\alpha,\beta)\in K_q\times \mu_q^{-1}(U_p^{F_q})$. It is clear that $ri=1$. We define a homotopy $H:ir\simeq 1$ as a composition of two linear homotopies. Any $\beta\in \mu_q^{-1}(U_p^{F_q})$ can be written as $\beta=t\beta_1+(1-t)\beta_2$ with $0<t\leq 1$, $\beta_1\in \K(U_p)$ and $\beta_2\in \K(X_p\smallsetminus U_p)$. Take $H_1((\alpha,\beta),s)=(\alpha,(1-s)\beta+s\beta_1)$. Since $\beta_1\in \K(U_p)$ which is a cone with apex $p$, we can define then $H_2((\alpha,\beta),s)=(\alpha,(1-s)\beta_1+sp)$.
\end{proof}

As a first immediate corollary we obtain Thomason's theorem for posets.

\begin{corollary}\label{thom}
Given a diagram of finite posets $X:P\to \P_{<\infty}$, there is a homotopy equivalence 
$$\thocolim \K X\to \K(\hocolim X).$$
\end{corollary}
\begin{proof}
Apply Theorem \ref{main} to the diagram morphism $\mu:\K X\to X$, where  $\mu_p:\K(X_p)\to X_p$ is the McCord map.
\end{proof}

Now we prove a kind of converse of Thomason's result, which relates the homotopy colimit of a diagram of simplicial complexes with the non-Hausdorff homotopy colimit of their face posets.

\begin{corollary}\label{antithom}
Let $K:P\to \mathcal{S}$ be a diagram of finite simplicial complexes. Then there is a weak equivalence $$\nu:\thocolim K\to \hocolim (\X K)^{op}$$
from the homotopy colimit of $K$ to the non-Hausdorff homotopy colimit of the diagram of the opposite of their face posets.
\end{corollary}
\begin{proof}
Apply Theorem \ref{main} to the diagram morphism induced by the natural weak equivalences $\nu_p:K_p\to \X(K_p)^{op}$ defined in \cite[Thm 11.3.2]{Ba1} and the beginning of the previous section.
\end{proof}

Since for every simplicial complex $L$, $\K(\X(L)^{op})=\K(\X(L))=L'$ (the barycentric subdivision of $L$), from Corollary \ref{thom} and Corollary \ref{antithom} we deduce the following.

\begin{corollary}
Let $K:P\to \mathcal{S}$ be a diagram of (unordered) finite simplicial complexes (and simplicial maps). Then $\thocolim K$ and $\thocolim K'$ are homotopy equivalent, where $K':P\to \mathcal{S}$ is the diagram of the barycentric subdivisions (of spaces and maps).
\end{corollary}

Note that the homotopy equivalence of the last corollary cannot be deduced directly from a diagram map between $K$ and $K'$ since, although the underlying topological spaces of $K_p$ and $K'_p$ are equal and the transition maps $f_{qp}$ and $f'_{qp}$ are (linearly) homotopic, in the context of unordered simplicial complexes there is no \it natural \rm homotopy equivalence  from the barycentric subdivision functor to the identity functor. Moreover, in general the homotopies $f_{qp}\simeq f'_{qp}$ cannot be taken coherently.

\begin{remark}
From Theorem \ref{main} one can easily deduce that if $X:P\to \mathcal{S}$ is a diagram of finite topological spaces, although $\thocolim X$ and $\hocolim X$ are very different, there is a weak homotopy equivalence $\thocolim X \to \hocolim X$ induced by the identity map. 
\end{remark}

Corollary \ref{antithom} allows one to apply the results of Section \ref{methods} to homotopy colimits of diagrams of polyhedra, by means of the weak equivalence with the non-Hausdorff homotopy colimits of the opposite of the face posets. In particular, the following analogues of Propositions \ref{up wp} and \ref{dbpgen} are valid for diagrams of simplicial complexes.

\begin{proposition}\label{up sc}
 Let $K:P\rightarrow \mathcal{S}$
be a $P$-diagram of finite simplicial complexes (and simplicial maps). If $p$ is a point of $P$ such that $\hat F_p$ is homotopically trivial (in particular if $p$ is an up beat point or an up weak point), then $\thocolim K\simeq \thocolim K_{|_{P\smallsetminus \{p\}}}$.
\end{proposition}
\begin{proof}
Consider the diagram $(\X K)^{op}:P\to \P_{<\infty}$, $(\X K)^{op}(p)=(\X(K_p))^{op}$, and apply Proposition \ref{up wp} and Corollary \ref{antithom}.
\end{proof}

\begin{proposition} \label{dbp sc}
Let $K:P\rightarrow \mathcal{S}$ be a $P$-
diagram of finite simplicial complexes (and simplicial maps).
If $p$ is a down beat point of $P$ dominated by $q$
and $f_{qp}$ is a homotopy equivalence, then
$$\thocolim K\simeq \thocolim K_{|_{P\smallsetminus \{p\}}}.$$
\end{proposition}
\begin{proof}
Consider the diagram $(\X K)^{op}:P\to \P_{<\infty}$ and apply Proposition \ref{dbpgen} and Corollary \ref{antithom}.
\end{proof}

Suppose that $K:P\to \mathcal{S}$ is a diagram of finite simplicial complexes such that all the transition  maps $f_{qp}:K_q\to K_p$ are homotopy equivalences (in particular, if $P$ is connected, all $K_p$ have the same homotopy type). In general, although the simplicial maps $f_{qp}:K_q\to K_p$ are homotopy equivalences,  the homotopy type of $K_p$ and the topology of $\K(P)$ don't determine the homotopy type of $\thocolim K$. We will show that if the indexing poset $P$ and the maps $f_{qp}$ satisfy nice conditions, then $\thocolim K\simeq K_p$ (for any $p\in P$).

\begin{corollary}\label{index contractible}
Let $K:P\rightarrow \mathcal{S}$ be a $P$-diagram of finite simplicial complexes (and simplicial maps). If $P$ is a contractible finite space (i.e. a dismantlable poset) and the transition maps $f_{qp}$ are homotopy equivalences, then $\thocolim K\simeq K_p$ (for any $p\in P$).
\end{corollary}

\begin{proof}
Since $P$ is a contractible finite space, there is a sequence $p_1, \cdots, p_n$ such that  $p_i$ is a beat point (up or down) of $P\smallsetminus \{p_1, p_2,
\cdots, p_{i-1}\}$ and $P\smallsetminus \{p_1,
p_2, \cdots, p_n\}=\{p\}$. Now apply recursively Propositions \ref{up sc} and \ref{dbp sc}.
\end{proof}

As we state at the beginning of the previous section, if $P$ is contractible (as a finite space) then its classifying space $\K(P)$ is contractible but the converse does not hold. In \cite{Ba1, BM1, BM2} there are various examples of non-contractible finite spaces $P$ with $\K(P)$ contractible. Suppose that the indexing poset $P$ can be reduced to a single point by removing $\gamma$-points (recall that $p$ is a $\gamma$-point if $\hat F_p$ or $\hat U_p$ is homotopically trivial). In that case the previous corollary is not longer valid since Proposition \ref{dbp sc} works only for down beat points (i.e. when $\hat U_p$ is a contractible finite space, not just homotopically trivial). However, one can impose  extra conditions on the maps $f_{qp}:K_q\to K_p$ in order to extend Proposition \ref{dbp sc} to $\gamma$-points, and Corollary \ref{index contractible} to a more general class of homotopically trivial posets. To this end we replace homotopy equivalences by \it contractible mappings. \rm  This class of maps was introduced by Cohen in \cite{Co}. A simplicial map $f:K\to L$ is called a contractible mapping if the preimage $f^{-1}(z)$ is contractible for every point $z$ in the underlying space of $L$. 
In \cite[Thm 11.1]{Co} Cohen proved that any contractible mapping $f:K\to L$ is a simple homotopy equivalence. In \cite[Thm 5.1]{Ba2} Barmak exhibited an alternative and simple proof of Cohen's result. From the proof of \cite[Thm 5.1]{Ba2} one can deduce the following.

\begin{proposition}[Barmak] \label{contractiblemapping}
Let $f:K\to L$ be a contractible mapping and let \linebreak $\X(f)^{op}:\X(K)^{op}\to \X(L)^{op}$ be the map induced in the opposite of their face posets. Then $(\X(f)^{op})^{-1}(U_{\sigma})$ is homotopically trivial for every $\sigma \in \X(L)^{op}$.
\end{proposition}

\begin{corollary}\label{dgp sc}
 Let $K:P\rightarrow \mathcal{S}$
be a $P$-diagram of finite simplicial complexes (and simplicial maps). If $p$ is a point of $P$ such that $\hat U_p$ is homotopically trivial and the transition maps $f_{qp}$ are contractible mappings for every $q\leq p$, then $\thocolim K\simeq \thocolim K_{|_{P\smallsetminus \{p\}}}$.
\end{corollary}
\begin{proof}
Consider the diagram $(\X K)^{op}:P\to \P_{<\infty}$ and follow the proof of Proposition \ref{dbp}, using that  $(\X(f_{qp})^{op})^{-1}(U_{\sigma})$ are homotopically trivial by Proposition \ref{contractiblemapping}.
\end{proof}

Corollary \ref{dgp sc} in combination with Proposition \ref{up sc} allows us to extend Corollary \ref{index contractible} to a more general class of (homotopically trivial) indexing posets, under the stronger assumption that the transition maps are contractible mappings.

\begin{corollary}
Let $K:P\rightarrow \mathcal{S}$ be a $P$-diagram of finite simplicial complexes. If the indexing poset $P$ can be reduced to a point by removing $\gamma$-points (in particular, if $P$ is collapsible), and the transition maps $f_{qp}$ are contractible mappings, then $\thocolim K\simeq K_p$ (for any $p\in P$).
\end{corollary}

\begin{example} If the transition maps of the following diagram of simplicial complexes are contractible mappings, its homotopy colimit is homotopy equivalent to any of the $K_p$. This is because the indexing poset is a collapsible (but non-contractible) finite space.

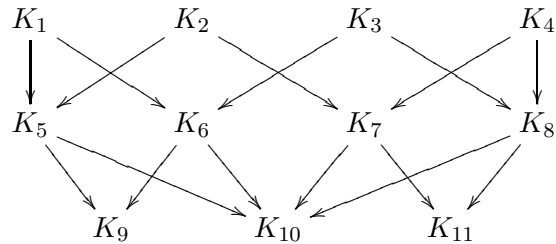
\begin{displaymath}
\xymatrix@C=10pt{ K_1 \ar@{->}[d] \ar@{->}[drr] & & K_2 \ar@{->}[lld] \ar@{->}[rrd] & & K_3 \ar@{->}[lld] \ar@{->}[rrd] & & K_4 \ar@{->}[lld] \ar@{->}[d]  \\
		K_5 \ar@{->}[dr] \ar@{->}[drrr] & & K_6 \ar@{->}[dl] \ar@{->}[dr] & & K_7 \ar@{->}[dl] \ar@{->}[dr] & & K_8 \ar@{->}[dlll] \ar@{->}[dl] \\
		& K_9 & & K_{10} & & K_{11} } 
\end{displaymath}

\end{example}

\bigskip

{\bf Acknowledgements.}
We are grateful to Jonathan Barmak for many useful discussions and suggestions during the preparation of this article. We also would like to thank Volkmar Welker and Tim Porter for useful comments.



\end{document}